\documentclass[reqno]{amsart}
\usepackage[backend=bibtex, sorting=none, bibstyle=alphabetic, citestyle=alphabetic, sorting=nyt,maxbibnames=99, giveninits=true, isbn=false, url=false, doi=false, eprint=false]{biblatex}  
\renewbibmacro{in:}{}
\bibliography{references.bib}
\usepackage{amssymb}
\usepackage{calrsfs}
\usepackage{graphics,graphicx}
\usepackage{enumerate}
\usepackage{enumitem}
\usepackage{url}
\usepackage{xcolor}
\usepackage[ruled, lined, linesnumbered, longend]{algorithm2e}
\usepackage{hyperref}
\usepackage[justification=centering]{caption}

\newtheorem{theorem}{Theorem}[section]
\newtheorem{lemma}[theorem]{Lemma}

\numberwithin{equation}{section}

\theoremstyle{remark}
\newtheorem*{remark}{Remark}

\title{An improved explicit estimate for $\zeta(1/2+it)$}
\author{Ghaith A. Hiary, Dhir Patel and Andrew Yang}
\date{July 2022}

\begin{document}

\begin{abstract}
    An explicit subconvex bound for the Riemann zeta function $\zeta(s)$ on the critical
    line $s=1/2+it$ is proved. Previous
    subconvex bounds 
    relied on an incorrect version of the Kusmin--Landau lemma. 
    After accounting for the needed correction in that lemma,
    we recover and improve the record explicit bound for $|\zeta(1/2 + it)|$. 
\end{abstract}

\maketitle

\section{Introduction}
    An important type of inequality in analytic number theory is upper bounds 
    on the exponential sum $S=|\sum_{a\le n <b} e^{2\pi i f(n)}|$ for some smooth phase function $f$.
    Suppose that $f'$ is monotonic and $\theta \le f'\le 1-\theta$ for some $\theta \in [0,1/2]$ 
    and throughout $[a,b)$. One would like to derive inequalities
    of the form $S\le A/\theta$, where $A$ is a constant. Among other applications, such results 
    are used to derive explicit upper bounds for $|\zeta(1/2 + it)|$. 
    
    Kusmin~\cite[p. 239]{kuzmin_sur_1927} writes that 
    inequalities for $S$ were first introduced by Vinogradov in 1916. 
    According
    to Landau~\cite[p. 21]{landau_ueber_1928}, van der Corput was the first to
    prove a bound on $S$ depending only on $\theta$, independent
    of the length of the interval $[a,b)$.
    Using results in \cite[p. 58]{corput_1921} and \cite[p.
    221]{landau_1926}, which rely on the Poisson summation formula, 
    it follows easily that $A=4$ is admissible. 
    Kusmin~\cite[p. 237]{kuzmin_sur_1927} gave a geometric argument to 
    improve this to $A=1$. He then immediately noted that it follows from his proof
    that $A=2/\pi$ is also admissible.
    Landau~\cite[p. 21]{landau_ueber_1928} refined the Kusmin
    bound to $S\le \cot(\pi \theta/2)$ and constructed examples
    for which the equality $S=\cot(\pi \theta/2)$ holds.

In \cite[Lemma 2]{cheng_explicit_2004}, 
the bound $S \le 1/(\pi \theta)+1$
was derived. 
Unfortunately, this bound misses an extra factor of $2$, i.e. the leading constant in the bound 
should be $2/\pi$ instead of $1/\pi$. This was recently pointed out by 
K.\ Ford and also discussed in a preprint by J.\ Arias de Reyna. 
Indeed, Kusmin~\cite[p. 237]{kuzmin_sur_1927} and
Landau~\cite[p. 21]{landau_ueber_1928} had shown that $A=2/\pi$ is sharp. Therefore, 
there is no hope of recovering the constant $1/\pi$ in \cite[Lemma 2]{cheng_explicit_2004}
in the general case, but only in special (though important) 
cases such as when $f$ is linear.

The incorrect 
constant $1/\pi$ has impacted 
all published explicit subconvex bounds on $|\zeta(1/2 + it)|$.
In particular, the bound $|\zeta(1/2 + it)| \le 0.63 t^{1/6}\log t$ 
in \cite[Theorem 1.1]{hiary_explicit_2016} is affected.
This bound 
relied on an explicit $B$ process (from the method of exponent pairs) 
that is derived in \cite[Lemma 3]{cheng_explicit_2004}. In turn, this explicit
$B$ process relied on the version of the Kusmin--Landau lemma in \cite[Lemma
2]{cheng_explicit_2004} from which the incorrect constant $1/\pi$ arises. 
After accounting for the correct constant $2/\pi$ in the
Kusmin--Landau lemma, the constant $0.63$ in \cite[Theorem 1.1]{hiary_explicit_2016}
increases substantially to $0.77$.
In other words, the revised bound becomes $|\zeta(1/2 + it)| \le 0.77 t^{1/6}\log t$ for
$t\ge 3$. 

Since the missing factor of $2$ in \cite[Lemma 3]{cheng_explicit_2004} is sizeable and the
method of proof in \cite{hiary_explicit_2016} is already optimized, 
our savings had to come in small quantities from multiple places.  

We specialize the phase function to our specific 
application of bounding zeta, and derive better explicit $B$ and 
$AB$ processes in lemmas \ref{imp_imp_second_deriv_test} and \ref{imp_imp_third_deriv_test},
as well as a generalized Kusmin-Landau lemma in Lemma~\ref{cglem}. 

Moreover, we are
exceedingly careful in treating boundary terms in the explicit $A$ and $B$
processes we derive. Boundary term may be all that arises
in the intermediate region where $t$ is too large for the
Riemann--Siegel--Lehman bound to be useful, but too small for the asymptotic
savings from the $A$ and $B$ processes to be realized. This intermediate region 
(bottleneck region) is the subject of
subsection \ref{mediumt}. 

Also, our treatment for large $t$ improves on 
\cite{hiary_explicit_2016} in one important 
aspect that enables reducing the coefficient of $t^{1/6}\log t$ considerably, 
as detailed at the end of
subsection \ref{larget}.

Put together, our main result is to recover and improve 
the constant $0.63$.

\begin{theorem}\label{main_theorem}
For $t \ge 3$, we have 
\[
|\zeta(1/2 + it)| \le 0.618 t^{1/6} \log t.
\]
\end{theorem}

Note that if one employs the Riemann–Siegel–Lehman bound
for any range of $t$ at all, then the leading constant cannot break $0.541$.
The constant we obtain is not too far from this barrier.

\subsection{Notation}
Throughout this work let $\|x\|$ denote 
the distance to the integer nearest to $x$, 
i.e. $\|x\| = \min_{n\in\mathbb{Z}}|x - n|$. 
We write $e(x) := \exp(2\pi i x)$.

\section{Required lemmas}

\begin{lemma}\label{hiarylem21}
If $t \ge 200$ and $n_1 := \lfloor \sqrt{t / (2\pi)}\rfloor$, then
\[
|\zeta(1/2 + it)| \le 2\left|\sum_{n = 1}^{n_1}n^{-1/2 + it}\right| + R(t)
\]
where $R(t) := 1.48 t^{-1/4} + 0.127 t^{-3/4}$.
\end{lemma}
\begin{proof}
    One starts with the Riemann--Siegel formula, and 
    applies the triangle inequality to the main
    sum and to the Gabcke remainder term, 
    bounding the latter by $R(t)$. 
    See Lemma 2.1 in \cite{hiary_explicit_2016}.
\end{proof}

\begin{lemma}[Riemann-Siegel-Lehman bound]\label{rsl_bound}
If $t \ge 200$, then 
\[
|\zeta(1/2 + it)| \le \frac{4t^{1/4}}{(2\pi)^{1/4}} - 2.08.
\]
\end{lemma}
\begin{proof}
    Follows from Lemma~\ref{hiarylem21} 
    on bounding the main sum for $n> 5$  by an integral 
    and using monotonicity to bound $R(t)$ for $t\ge
    200$. See Lemma 2.3 in \cite{hiary_explicit_2016}.
\end{proof}

\begin{lemma}[Generalised Cheng-Graham lemma]\label{cglem}
Let $f(x)$ be a real-valued function with a monotonic and continuous derivative on $[a, b)$, satisfying 
\[\ell+U^{-1} \le f'(x) \le \ell+1 - V^{-1},\qquad x\in[a, b),\]
for some $U, V > 1$ and some integer $\ell$. Then,
\[
\left|\sum_{a \le n < b}e(f(n))\right| \le \frac{U + V}{\pi}.
\]
\end{lemma}
\begin{proof}The proof follows from a similar argument to Lemma 2.1 in Patel
    \cite{PATEL2022} and Kuzmin \cite{kuzmin_sur_1927} and Landau
    \cite{landau_ueber_1928}. See Section \ref{cglemproof} for details.
This generalized lemma is essential
to the cases considered at the end of the proof of
Lemma~\ref{imp_imp_second_deriv_test}. 
\end{proof}

\begin{lemma}[Weyl differencing]\label{weyldifflem}
Let $f(n)$ be a real-valued function and $L$ and $M$ positive integers. Then
\[
\left|\sum_{n = N + 1}^{N + L}e(f(n))\right|^2 \le \left(\frac{L + M - 1}{M}\right)\left(L + 2\sum_{m = 1}^{M}\left(1 - \frac{m}{M}\right)|s'_m(L)|\right)
\]
    where if $m < L$, then 
\[
s_m'(L) := \sum_{n = N + 1}^{N + L - m}e(f(r + m) - f(r)),
\]
    and if $m\ge L$, then $s_m'(L)=0$.
\end{lemma}
\begin{proof}
    See Cheng and Graham \cite[Lemma 5]{cheng_explicit_2004} and Platt and
    Trudgian \cite{platt_improved_2015}. 
        Lemma 5 in \cite{cheng_explicit_2004}
    appears with $\max_{L_1\le L}|s'_m(L_1)|$ instead of $|s'_m(L)|$.
    It was pointed out in \cite{hiary_explicit_2016}, though, that 
    one can remove the $\max$ by using 
    the more precise form presented at the bottom of page 1273 in 
    \cite{cheng_explicit_2004}.
\end{proof}

\begin{lemma}[Improved second derivative test]\label{imp_imp_second_deriv_test}
        Let $m, r, K$ be positive integers, and let $t$ and $K_0$ be a positive
        numbers. 
        Suppose $K \ge K_0 > 1$ and $m < K$. Let
    
\begin{align*}
g(x) := \frac{t}{2\pi}\log\left(1 + \frac{m}{rK + x}\right),\qquad W :=
    \frac{\pi(r + 1)^3K^3}{t},\qquad\lambda := \frac{(1 + r)^3}{r^3}.
\end{align*}
    So, by construction, 
\[
    \frac{m}{W} \le |g''(x)| \le \frac{m\lambda}{W}, \qquad (0\le
    x\le K-m).
\]
    For each positive integer $L \le K$, and each positive integer $m<L$,
\[
    \left|\sum_{n=0}^{L - 1 - m}e(g(n))\right| \le \frac{4\mu K}{\sqrt{\pi
    W}}m^{1/2} + \frac{\mu K}{W}m + 4\sqrt{\frac{W}{\pi}}m^{-1/2} + 2 - \frac{4}{\pi},
\]
where 
\[
\mu := \frac{1}{2}\lambda^{2/3}\left(1 + \frac{1}{(1 - K_0^{-1})\lambda^{1/3}}\right).
\]
        If $m\ge L$ or $L=1$, then the sum on the left-side is zero and the bound still holds.  
\end{lemma}

\begin{proof}
    See Section \ref{second_deriv_proof}. 
\end{proof}

\begin{remark}
    We have $g(x) = f(x+m)-f(x)$, where
    $f(x)$ is defined in Lemma~\ref{imp_imp_third_deriv_test}.
    Thus, $g(x)$ is the phase function that arises
    after we apply the Weyl
    differencing from Lemma~\ref{weyldifflem} to a contiguous portion of the main sum in
    Lemma~\ref{rsl_bound}.
\end{remark}

\begin{lemma}[Improved third derivative test]\label{imp_imp_third_deriv_test}
Let $r$ and $K$ be positive integers, and let $t$ and $K_0$ be positive numbers.
    Suppose $K \ge K_0 > 1$. Let 
\begin{align*}
f(x) := \frac{t}{2\pi}\log\left(rK + x\right).
\end{align*}
Furthermore let $W$ and $\lambda$ be as defined in Lemma \ref{imp_imp_second_deriv_test}, so that 
\[
\frac{1}{W} \le |f'''(x)| \le \frac{\lambda}{W}.
\]
For each positive integer $L \le K$, each integer $N$,
and any $\eta > 0$,
\[
\left|\sum_{n = N + 1}^{N + L}e(f(n))\right|^2 \le \left(\frac{K}{W^{1/3}} + \eta\right)(\alpha K + \beta W^{2/3})
\]
where
\[
\alpha := \frac{1}{\eta} + \frac{\eta\mu}{3W^{1/3}} + \frac{\mu}{3W^{2/3}} + \frac{32\mu}{15\sqrt{\pi}}\sqrt{\eta + W^{-1/3}},
\]
\[
    \beta := \frac{32}{3\sqrt{\pi \eta}} + \left(2 - \frac{4}{\pi}\right)\frac{1}{W^{1/3}}.
\]
Here, $\mu$ is defined as in Lemma \ref{imp_imp_second_deriv_test}.
\end{lemma}
\begin{proof}
See Section \ref{third_deriv_proof}.
\end{proof}

\begin{remark}
    Versions of these derivative tests
    for a general phase function, as well as 
several other derivative tests, can be found in \cite{PATEL2022}.
\end{remark}

\section{Proof of Theorem 1}\label{thm1_proof}
We divide the proof into four regions. 
\subsection{Proof for $3 \le t < 200$}
In this range we rely on the interval-arithmetic computations carried out in Hiary \cite{hiary_explicit_2016}, which established 
\[
|\zeta(1/2 + it)| \le 0.595 t^{1/6}\log t
\]
for $3 \le t < 200$.
\subsection{Proof for $200 \le t < 5.5\cdot 10^7$} For this region we use the
Riemann-Siegel-Lehman formula combined with the triangle inequality. Firstly, in
preparation for using Lemma~\ref{rsl_bound}, we note
\[
\frac{4t^{1/4}}{(2\pi)^{1/4}} - 2.08 < 0.592 t^{1/6}\log t,\qquad 200 \le t \le
10^7.
\]
This can be seen by verifying that the difference of the two sides is unimodal
(monotonically increasing then monotonically decreasing),
and so it suffices to check that the difference is positive at the endpoints
$200$ and $10^7$. Hence, our main theorem follows from Lemma \ref{rsl_bound} for
that range of $t$.

Assume now that $10^{7}\le
t < 5.5\cdot 10^7$. We follow a similar argument to Lemma 2.3 in \cite{hiary_explicit_2016}. By Lemma \ref{hiarylem21},
\[
|\zeta(1/2 + it)| \le  2\left|\sum_{n = 1}^{n_1}n^{-1/2 + it}\right| + 1.48 t_0^{-1/4} + 0.127 t_0^{-3/4},\qquad t \ge t_0
\]
where $n_1 = \lfloor\sqrt{t/(2\pi)}\rfloor$ and, in this subsection, $t_0 = 10^7$. 
Next, we observe that
if $h$ is a real-valued function such that $h''(x) > 0$ for $a-1/2\le x\le
b+1/2$, then by Jensen's inequality
\begin{equation}\label{jensenineq}
\sum_{n = a}^b h(n) = \sum_{n = a}^b h\left(\int_{n - \frac{1}{2}}^{n + \frac{1}{2}}x\text{d}x\right) \le \sum_{n = a}^b\int_{n - \frac{1}{2}}^{n + \frac{1}{2}}h(x)\text{d}x = \int_{a - \frac{1}{2}}^{b + \frac{1}{2}}h(x)\text{d}x.
\end{equation}
Therefore, using $h(x) = x^{-1/2}$ and the fact that $n_1 \ge \lfloor\sqrt{t_0 / (2\pi)}\rfloor = 1261$,
\begin{align*}
\left|\sum_{n = 1}^{n_1}n^{-1/2+it}\right| &\le \sum_{n = 1}^{1261}\frac{1}{\sqrt{n}} + \int_{1260.5}^{n_1 + \frac{1}{2}}\frac{\text{d}x}{\sqrt{x}} \\
&\le 69.575 + 2\sqrt{n_1} - 2\sqrt{1260.5} + \frac{1}{\sqrt{n_1}}\int_{n_1}^{n_1 + \frac{1}{2}}\text{d}x\\
&\le 2\sqrt{n_1} + \frac{1}{2\sqrt{1261}} - 1.432\\
&= \frac{2t^{1/4}}{(2\pi)^{1/4}} - 1.417.
\end{align*}
However, $1.48 t_0^{-1/4} + 0.127 t_0^{-3/4} \le 0.027$, so 
\[
|\zeta(1/2 + it)| \le \frac{4t^{1/4}}{(2\pi)^{1/4}} - 2.807 \le 0.618 t^{1/6}\log t
\]
for $10^7 \le t < 5.5\cdot 10^7$. One need verify the last
inequality at the endpoints $10^7$ and $5.5\cdot 10^7$ since the difference of
the two sides is monotonic in between. 

\subsection{Proof for $5.5\cdot 10^7 \le t < 10^{12}$}\label{mediumt}
In this range of $t$ we use the following modified third-derivative test in 
Lemma~\ref{imp_imp_third_deriv_test}.

Throughout this subsection, let $r_0 > 1$ be an integer and 
 $\frac{1}{3} < \phi < \frac{1}{2}$ be a constant, both to be chosen later. 
Furthermore, let 
$$K := \lceil t^{\phi}\rceil, \qquad
n_1 := \lfloor\sqrt{t /2\pi}\rfloor,\qquad 
R := \lfloor n_1 / K\rfloor.$$
Also, let $t_0\ge 5.5\cdot 10^7$, to be chosen later, and suppose $t\ge t_0$. 
By Lemma \ref{hiarylem21}, we have 
\begin{align*}
    |\zeta(1/2 + it)| &\le T_1 + T_2,
\end{align*}
where 
\begin{align*}
    T_1 &:= 2\sum_{1\le n < r_0K} \frac{1}{\sqrt{n}} + 1.48t_0^{-1/4} +
    0.127t_0^{-3/4},\\
    T_2 &:= 2\sum_{r = r_0}^{R - 1}\left|\sum_{rK \le n < (r +
    1)K}\frac{e^{it\log n}}{\sqrt{n}}\right| + 2\left|\sum_{RK \le n \le n_1}\frac{e^{it\log n}}{\sqrt{n}}\right|.
\end{align*}
Define
\[
I_{\phi}(r_0, t_0) := 2\sum_{n = 1}^{\lceil r_0t_0^{\phi}\rceil - 1}\frac{1}{\sqrt{n}} - 4\sqrt{\lceil r_0 t_0^{\phi}\rceil - 1}.
\]
Then, noting that $K \le t^{\phi} + 1$, and following the arguments in Hiary \cite{hiary_explicit_2016},
\begin{align}
T_1 &\le 4\sqrt{r_0K} + I_{\phi}(r_0, t_0) + 1.48t_0^{-1/4} +
    0.127t_0^{-3/4}\notag\\
&\le 4\sqrt{r_0\left(1 + t_0^{-\phi}\right)}t^{\phi / 2} + I_{\phi}(r_0, t_0) + 1.48t_0^{-1/4} + 0.127t_0^{-3/4},\label{T1bound}
\end{align}
for all $t \ge t_0$. Meanwhile, by partial summation,
\begin{align}
    \sum_{r = r_0}^{R - 1}\left|\sum_{rK \le n < (r +1)K}\frac{e^{it\log
    n}}{\sqrt{n}}\right| \le \sum_{r = r_0}^{R -
    1}\frac{1}{\sqrt{rK}}\max_{\Delta \le K}\left|\sum_{k = 0}^{\Delta -
    1}e^{it\log (rK + k)}\right|.
\end{align}
So, since $n_1 \le (R + 1)K$,
\begin{equation}\label{final_piece_bound}
    \left|\sum_{RK \le n \le n_1}\frac{e^{it\log n}}{\sqrt{n}}\right| \le
    \frac{1}{\sqrt{RK}}\max_{\Delta \le K}\left|\sum_{k = 0}^{\Delta -
    1}e^{it\log (RK + k)}\right|,
\end{equation}
it follows
\[
T_2 \le 2\sum_{r = r_0}^R \frac{1}{\sqrt{rK}}\max_{\Delta \le K}\left|\sum_{k = 0}^{\Delta - 1}e^{it\log(rK + k)}\right|.
\]
We apply Lemma \ref{imp_imp_third_deriv_test} with $K_0 := t_0^{\phi} > 1$ to
obtain for any $\eta>0$, 
\begin{align}
T_2 &\le 2\sum_{r = r_0}^R \frac{1}{\sqrt{rK}}\sqrt{\left(\frac{K}{W^{1/3}} +
    \eta\right)(\alpha K + \beta W^{2/3})} \label{S_ineq1} \\
&= \frac{2t^{1/6}}{\pi^{1/6}}\sum_{r = r_0}^R\frac{1}{\sqrt{r(r +
    1)}}\sqrt{\alpha + \frac{\eta\alpha W^{1/3}}{K} + \frac{\beta W^{2/3}}{K} +
    \frac{\eta\beta W}{K^2}},\label{S_ineq}
\end{align}
where $W$, $\lambda$, $\alpha$, $\beta$, and $\mu$
are defined in Lemma \ref{imp_imp_third_deriv_test}.

Using the same trick with the Jensen inequality as in the previous section,  
we observe that
\begin{align}
\sum_{r = r_0}^R\frac{1}{\sqrt{r(r + 1)}} &\le \sum_{r = r_0}^R\int_{r -
    \frac{1}{2}}^{r + \frac{1}{2}}\frac{\text{d}x}{\sqrt{x(x + 1)}} =\int_{r_0 -
    \frac{1}{2}}^{R + \frac{1}{2}}\frac{\text{d}x}{\sqrt{x(x + 1)}} = \left[2\sinh^{-1}\sqrt{x}\right]_{r_0 - \frac{1}{2}}^{R + \frac{1}{2}}\notag\\
&= 2\sinh^{-1}\sqrt{R + \frac{1}{2}} - 2\sinh^{-1}\sqrt{r_0 - \frac{1}{2}}\label{rRbound}
\end{align}
where the inequality follows from using $h(x) = 1/\sqrt{x(x + 1)}$ in \eqref{jensenineq}. In addition,
\begin{equation}\label{arcsinhformula}
2\sinh^{-1}\sqrt{x} = \log x + 2\log\left(1 + \sqrt{1 + \frac{1}{x}}\right)
\end{equation}
for $x\ge 0$, and
\[
R \ge R_0 := \left\lceil \frac{\sqrt{t_0 / 2\pi} - 1}{t_0^{\phi} + 1} -
1\right\rceil,
\]
hence 
\begin{align*}
2\sinh^{-1}\sqrt{R + \frac{1}{2}} &= \log R + \log \left(1 + \frac{1}{2R}\right) + 2\log\left(1 + \sqrt{1 + \frac{1}{R + \frac{1}{2}}}\right)\\
&\le \log R + J(R_0),
\end{align*}
where 
\[
J(R_0) := \log\left(1 + \frac{1}{2R_0}\right) + 2\log\left(1 + \sqrt{1 + \frac{1}{R_0 +  \frac{1}{2}}}\right).
\]
Next, let 
\[
\rho := \frac{1}{\sqrt{2}\pi^{1/6}}\left(1 + \frac{1}{R}\right).
\]
Since $R \le t^{1/2 - \phi}/\sqrt{2\pi}$ and $K \le t^{\phi} + 1$, we have
\begin{equation}\label{W_bound}
W^{1/3} = \frac{\pi^{1/3}(r + 1)K}{t^{1/3}} \le \frac{\pi^{1/3}(R + 1)(t^{\phi} + 1)}{t^{1/3}} \le \rho t^{1/6}\left(1 + \frac{1}{t^{\phi}}\right)
\end{equation}
and 
\begin{equation}\label{WK_bound}
\frac{W^{1/3}}{K} = \frac{\pi^{1/3}(r + 1)K}{Kt^{1/3}} \le \frac{\pi^{1/3}R}{t^{1/3}}\left(1 + \frac{1}{R}\right) \le \rho t^{1/6 - \phi}.
\end{equation}
Therefore,

\begin{equation}\label{beta_terms_bound}
    \begin{split}
        &\frac{\eta\alpha W^{1/3}}{K} \le \frac{\eta\alpha \rho}{t^{\phi - 1/6}},\\
        &\frac{\beta W^{2/3}}{K} = \beta W^{1/3}\frac{W^{1/3}}{K} \le \frac{\beta\rho^2\left(1 + t^{-\phi}\right)}{t^{\phi - 1/3}},\\
        &\frac{\eta\beta W}{K^2} = \eta\beta W^{1/3}\left(\frac{W^{1/3}}{K}\right)^2 \le \frac{\eta\beta\rho^3(1 + t^{-\phi})^2}{t^{2\phi - 1/2}}.
    \end{split}
\end{equation}
We apply the above inequalities to \eqref{S_ineq}, together with the inequality
\[
\rho_0 := \frac{1}{\sqrt{2}\pi^{1/6}}\left(1 + \frac{1}{R_0}\right) \ge \rho,
\]
to obtain
\begin{equation}
    T_2 \le 
    \frac{2t^{1/6}}{\pi^{1/6}}\sum_{r = r_0}^R\frac{1}{\sqrt{r(r + 1)}}
    \sqrt{\alpha + \frac{\eta\alpha \rho_0}{t^{\phi -
    1/6}} + \frac{\beta\rho_0^2\left(1 + t^{-\phi}\right)}{t^{\phi - 1/3}} +
    \frac{\eta\beta\rho_0^3(1 + t^{-\phi})^2}{t^{2\phi - 1/2}}}
\end{equation}
We observe since $\lambda$ is monotonically
decreasing with $r$, $\mu$ is decreasing with $r$. 
Since, in addition, $W$ is monotonically increasing with $r$, 
$\alpha$ and $\beta$ are both decreasing with $r$.
Denoting the values of $W$, $\alpha$ and $\beta$ at $r=r_0$ by $W_0$, $\alpha_0$ and
$\beta_0$, we see that $W\ge W_0$, $\alpha\le \alpha_0$ and $\beta\le \beta_0$.
Therefore, using \eqref{rRbound} and the subsequent estimates, we obtain
\begin{align}
    T_2 &\le \frac{2t^{1/6}}{\pi^{1/6}}\left[\,\left(\frac{1}{2} - \phi\right)\log
    t - \log \sqrt{2\pi} + J(R_0) - 2\sinh^{-1}\sqrt{r_0 -
    \frac{1}{2}}\,\,\right]\kappa_{\phi}
\end{align}
where, as we also have $\phi>1/3$ and $t\ge t_0$, $\kappa_{\phi}$ may be taken to be
\begin{equation}
    \kappa_{\phi}:=\sqrt{\alpha_0 + \frac{\eta\alpha_0 \rho_0}{t_0^{\phi -
    1/6}} + \frac{\beta_0\rho_0^2\left(1 + t_0^{-\phi}\right)}{t_0^{\phi - 1/3}} +
    \frac{\eta\beta_0\rho_0^3(1 + t_0^{-\phi})^2}{t_0^{2\phi - 1/2}}}.
\end{equation}
Explicitly, combining with \eqref{T1bound}, yields
\[
|\zeta(1/2 + it)| \le a_1 t^{1/6}\log t + a_2t^{1/6} + a_3t^{\phi / 2} + a_4, \qquad t\ge t_0,
\]
where, 
\begin{equation}\label{final_mediumt_bound}
    \begin{split}
        a_1 &:= \frac{1 - 2\phi}{\pi^{1/6}}\kappa_{\phi},\\
        a_2 &:= - \frac{2}{\pi^{1/6}}\left[\log \sqrt{2\pi} - J(R_0) +
        2\sinh^{-1}\sqrt{r_0 - \frac{1}{2}}\right]\kappa_{\phi},\\
        a_3 &:= 4\sqrt{r_0\left(1 + t_0^{-\phi}\right)},\\
        a_4 &:= I_{\phi}(r_0, t_0) + 1.48t_0^{-1/4} + 0.127t_0^{-3/4}.
    \end{split}
\end{equation}
We choose $\phi = 0.3414$, $\eta = 1.8$, $r_0 = 4$ and $t_0 = 5.5\cdot 10^7$.
This choice of $r_0$ is valid since it satisfies $r_0\le R_0$.
Also, in view of the chosen values for $t_0$ and $\phi$,
the inequality $W_0\ge \pi (r_0+1)^3 \lceil
K_0\rceil^3/t_0$, for $t\ge t_0$, is valid. We use this inequality to simplify
the bounds for $\alpha_0$ and $\beta_0$, 
considering they are both monotonically decreasing with $W_0$.
Together, we obtain
\begin{align*}
|\zeta(1/2 + it)| &\le 0.59289 t^{1/6} \log t - 8.0314 t^{1/6} + 8.0092
    t^{0.1707} - 2.8796,\qquad t \ge 5.5\cdot 10^7\\
&\le 0.618 t^{1/6}\log t,\qquad \text{for } 5.5\cdot 10^{7} \le t < 10^{8}.
\end{align*}
Similarly, choosing $\phi = 0.3414$, $\eta = 1.8$, $r_0 = 4$ and $t_0 = 10^8$, we obtain
\begin{align*}
|\zeta(1/2 + it)| &\le 0.58589 t^{1/6} \log t - 8.0115 t^{1/6} + 8.0075 t^{0.1707} - 2.8843,\qquad t \ge 10^{8}\\
&\le 0.618 t^{1/6}\log t,\qquad \text{for } 10^{8} \le t < 8.5 \cdot 10^{10}.
\end{align*}
Finally, choosing $\phi = 0.3414$, $\eta = 1.8$, $r_0 = 4$ and $t_0 = 8.5\cdot 10^{10}$, we obtain
\begin{align*}
|\zeta(1/2 + it)| &\le 0.55305 t^{1/6} \log t - 7.8629 t^{1/6} + 8.0008
    t^{0.1707} - 2.9111,\qquad t \ge 8.5\cdot 10^{10}\\
&\le 0.618 t^{1/6}\log t,\qquad \text{for } 8.5\cdot 10^{10} \le t < 10^{12},
\end{align*}
hence the desired result holds for $5.5\cdot 10^7 \le t < 10^{12}$. 

\subsection{Proof for $t \ge 10^{12}$}\label{larget} 
For the region $t \ge 10^{12}$ we use a
similar method as the previous subsection, but with $\phi = 1/3$. Analogously to
before, let $K = \lceil t^{1/3}\rceil$, $n_1 = \lfloor\sqrt{t / 2\pi}\rfloor$,
$R = \lfloor n_1 / K\rfloor$ and, this time, let $t_0 = 10^{12}$. Suppose $t\ge
t_0$.

We bound $T_2$ differently, by splitting the square root in \eqref{S_ineq} as follows. Let 
\begin{equation}\label{Beta_r_defn}
\mathcal{B}_r := \alpha + \frac{\eta\alpha W^{1/3}}{K} + \frac{\beta W^{2/3}}{K} + \frac{\eta\beta W}{K^2},
\end{equation}
so that, recalling \eqref{S_ineq}, we have
\begin{equation}\label{T2_ineq}
T_2 \le \frac{2t^{1/6}}{\pi^{1/6}}\sum_{r = r_0}^R\sqrt{\frac{\mathcal{B}_r}{r(r + 1)}}.
\end{equation}
Substituting $\phi = 1/3$ into \eqref{beta_terms_bound}, and noting that $t \ge
t_0$ and $\rho \le \rho_0$, we deduce
\begin{equation}\label{beta_terms_bound2}
\frac{\eta\alpha W^{1/3}}{K} \le \eta\alpha t_0^{-1/6}\rho_0, \qquad \frac{\eta\beta W}{K^2} \le \eta\beta t_0^{-1/6}\rho_0^3(1 + t_0^{-1/3})^2.
\end{equation}
Next, using the definition of $W$, and since $K \le t^{1/3} + 1$, $r \ge r_0$ and $t \ge t_0$,
\begin{equation}\label{beta_term2_bound}
\frac{1}{r(r + 1)}\frac{\beta W^{2/3}}{K} = \frac{\beta}{r(r + 1)}\frac{\pi^{2/3}(r + 1)^2K}{t^{2/3}} \le \frac{\beta \pi^{2/3}(1 + r_0^{-1})(1 + t_0^{-1/3})}{t^{1/3}}.
\end{equation}
Furthermore, using $\sqrt{x + y}\le \sqrt{x} + \sqrt{y}$, valid 
for any nonnegative numbers $x$ and
$y$, and combining \eqref{Beta_r_defn}, \eqref{beta_terms_bound2} and
\eqref{beta_term2_bound}, as well as the observation $\alpha\le \alpha_0$ and $\beta\le
\beta_0$ which follows by the same reasoning as in subsection \ref{mediumt},  we thus see
\begin{align}
    \sqrt{\frac{\mathcal{B}_r}{r(r + 1)}} &\le \sqrt{\frac{1}{r(r +
    1)}\left(\alpha + \frac{\eta\alpha W^{1/3}}{K} +  \frac{\eta\beta
    W}{K^2}\right)} + \sqrt{\frac{1}{r(r + 1)}\frac{\beta W^{2/3}}{K}}\notag\\
    &\le \frac{\kappa'_0}{\sqrt{r(r + 1)}} 
    + \frac{1}{t^{1/6}}\sqrt{\beta_0\pi^{2/3}(1 + r_0^{-1})(1 + t_0^{-1/3})},\label{Beta_r_bound}
\end{align}
where 
\[
\kappa'_0 := \sqrt{\alpha_0 + \eta\alpha_0 t_0^{-1/6}\rho_0 + \eta\beta_0 t_0^{-1/6}\rho_0^3(1 + t_0^{-1/3})^2}.
\]
We execute the sum over $r$ using \eqref{rRbound} and \eqref{arcsinhformula}.
To bound the resulting terms, we appeal to the estimate
\[R + \frac{1}{2} \ge \left\lceil \frac{\sqrt{t_0 / 2\pi} - 1}{t_0^{1/3} + 1} -
1\right\rceil + \frac{1}{2} \ge 39.5,\]
valid for $t \ge t_0$, and implying that
\begin{equation}\label{asinh_bound}
2\sinh^{-1}\sqrt{R + \frac{1}{2}} = \log \left(R + \frac{1}{2}\right) +
    2\log\left(1 + \sqrt{1 + \frac{1}{39.5}}\right)\le \log R + 1.412,
\end{equation}
in that range of $t$.
Note, in addition, 
that $R \le t^{1/6}/\sqrt{2\pi}$. Therefore, combined with 
\eqref{asinh_bound} we obtain
\begin{equation}\label{r_sum_bound}
    \sum_{r=r_0}^R \frac{1}{\sqrt{r(r+1)}} \le \frac{1}{6}\log t + \omega_0
\end{equation}
where
\begin{equation}
\omega_0 :=   -\log\sqrt{2\pi}
    +1.412-2\sinh^{-1}\sqrt{r_0-1/2}.
\end{equation}
Using this together with \eqref{Beta_r_bound},
it follows on recalling \eqref{T2_ineq} that
\begin{equation}\label{T2_bound_2}
    \begin{split}
        T_2 
        &\le \frac{2t^{1/6}}{\pi^{1/6}}\left(\frac{\log t}{6} +\omega_0\right)\kappa_0' 
    + \frac{\sqrt{2}t^{1/6}}{\pi^{1/3}}\sqrt{\beta_0(1 + r_0^{-1})(1 +
        t_0^{-1/3})},
    \end{split}
\end{equation}
for $t \ge t_0$.

As for $T_1$, we bound it the same way as in \eqref{T1bound} but with $\phi=1/3$,
which gives
\begin{equation}\label{T1_bound_2}
T_1 \le 4\sqrt{r_0\left(1 + t_0^{-1/3}\right)}t^{1/6} + I_{\frac{1}{3}}(r_0, t_0) + 1.48t_0^{-1/4} + 0.127t_0^{-3/4}.
\end{equation}
for all $t \ge t_0$. 

Finally, combining \eqref{T2_bound_2} and \eqref{T1_bound_2}, we arrive at
\[
|\zeta(1/2 + it)| \le b_1 t^{1/6}\log t + b_2 t^{1/6} + b_3, \qquad\text{for all }t \ge t_0,
\]
where 
\begin{equation*}
    \begin{split}
        b_1 &:= \frac{\kappa'_0}{3\pi^{1/6}},\\
        b_2 &:= 4\sqrt{r_0\left(1 + t_0^{-1/3}\right)} 
        +\frac{2}{\pi^{1/6}}\omega_0 \kappa'_0
        + \frac{\sqrt{2}}{\pi^{1/3}}\sqrt{\beta_0(1 + r_0^{-1})(1 + t_0^{-1/3})},\\
        b_3 &:= I_{\frac{1}{3}}(r_0, t_0) + 1.48t_0^{-1/4} + 0.127t_0^{-3/4}.
    \end{split}
\end{equation*}
Choosing $r_0 = 4$ and $\eta = 1.6$, and
using the inequality $W_0\ge \pi (r_0+1)^3$ 
to remove remaining dependence of $\alpha_0$ and $\beta_0$ on
$t$, yields
\begin{align*}
|\zeta(1/2 + it)| &\le 0.478013 t^{1/6}\log t + 3.853165 t^{1/6} - 2.914229\\
    &\le 0.618 t^{1/6}\log t, \qquad (t\ge t_0),
\end{align*}
as required.

We point out 
one of main reasons for the 
improvement over \cite{hiary_explicit_2016}
 obtained in this subsection. 
After we invoke the explicit $AB$ process
    from Lemma~\ref{imp_imp_third_deriv_test} to arrive at \eqref{S_ineq1},
we pay greater attention to the cross term $(K/W^{1/3})(\beta
W^{2/3})=\beta K W^{1/3}$. 
Specifically, we arrange for this cross term 
to contribute to the coefficient of 
    $t^{1/6}$ in the overall bound in \eqref{T2_bound_2}, 
    rather than to the coefficient of
$t^{1/6}\log t$, as done in \cite{hiary_explicit_2016}.
This saves a factor of $\log t$ from the contribution of this term, which is
a considerable saving. 

\section{Proof of Lemma \ref{cglem}}\label{cglemproof}
The proof proceeds similarly to Lemma 2 in \cite{cheng_explicit_2004} and Lemma
2.1 in Patel \cite{PATEL2022}, with only a few differences. We include the
complete proof here for convenience.

Let $f$ be a function satisfying the conditions of Lemma \ref{cglem}, that is, $f$ has a continuous and monotonic derivative on $[a, b)$ and
\begin{equation}\label{f_cond}
\ell + U^{-1} \le f'(x) \le \ell + 1 - V^{-1},\qquad x\in[a, b).
\end{equation}
for some $U, V > 1$ and some integer $\ell$. 
We may assume that $\ell = 0$, i.e. that $U^{-1}\le f'(x) \le 1 - V^{-1}$, since
\[
\left|e(-\ell n)\sum_{a \le n < b}e(f(n))\right| = \left|\sum_{a \le n <
b}e(f(n) - \ell n)\right|.
\]
We may also assume that $f'(x)$ is increasing, since we may replace $f(x)$ with
$-f(x)$ without changing the magnitude of the sum. 

Now, define $g(x) := f(x + 1) - f(x)$. 
Since by assumption $f'(x)$ is increasing in $x$ over $x\in [a,b)$,
$g'(x) = f'(x + 1) - f'(x) \ge
0$ over $x\in [a,b-1)$. 
Therefore, $g(x)$ is increasing in $x$ over that interval.
Furthermore, by the mean-value theorem, $g(x) = (x + 1 - x)f'(\xi) = f'(\xi)$
for some $\xi\in(x, x + 1)$. Therefore, 
\begin{equation}\label{g_bounds}
U^{-1} \le f'(x) \le f'(\xi) = g(x) \le f'(x + 1) \le 1 - V^{-1},
\end{equation}
for every $x \in [a, b - 1)$.
Thus, for instance, $0<g(x)<1$ over $x\in [a,b-1)$. 

Next, let 
\begin{equation}
G(n) := \frac{1}{1 - e(g(n))} = \frac{1 + i\cot(\pi g(n))}{2},
\end{equation}
so that
\begin{equation}\label{Gn_eqn}
\begin{split}
   G(n)\left[e(f(n)) - e(f(n + 1))\right] &= \frac{e(f(n)) - e(f(n + 1))}{1 - e(g(n))}\\
   &= \frac{e(f(n))[1-e(g(n))]}{1 - e(g(n))} \\
   &= e(f(n)), 
\end{split}
\end{equation}
and also
\begin{equation}\label{G_diff}
G(n) - G(n - 1) = \frac{\cot(\pi g(n - 1)) - \cot(\pi g(n))}{2i}.
\end{equation}

Let $L = \lceil a\rceil$. If $b$ is not an integer, let $M = \lfloor b \rfloor$.
If $b$ is an integer, let $M=b-1$. 
In either cases, the summation over $n\in [a,b)$ is the same as the summation
over $n\in [L,M]$, and $g'(x)$ is increasing in $x\in [L,M-1]$.

Suppose that $M = L$, so there is only one term in the sum. 
Then, by the trivial bound, we have
\begin{equation}
\left|\sum_{a \le n < b}e(f(n))\right| \le 1 \le \frac{U + V}{\pi}.
\end{equation}
Here we use the fact that 
the condition \eqref{f_cond} implies that $U^{-1} \le 1 - V^{-1}$, so
by the inequality of harmonic and arithmetic means, 
\begin{equation}
U + V \ge \frac{4}{U^{-1} + V^{-1}} \ge 4.
\end{equation}
Therefore, the result of the lemma follows when $M=L$.

Next, suppose $M - L \ge 1$. By \eqref{Gn_eqn}, and after a few
rearrangements,
\begin{align*}
\left|\sum_{n = L}^M e(f(n))\right| &= \left|\sum_{n = L}^{M - 1}G(n)[e(f(n)) - e(f(n + 1))] + e(f(M))\right|\\
&= \Bigg|\sum_{n = L + 1}^{M - 1}e(f(n))(G(n) - G(n - 1))  \\
&\qquad\qquad + e(f(L))G(L) + e(f(M))(1 - G(M - 1))\Bigg|\\
&\le \left|e(f(L))G(L)\right| + |e(f(M))(1 - G(M - 1))|\\
&\qquad\qquad + \sum_{n = L + 1}^{M - 1}\left|e(f(n))(G(n) - G(n - 1))\right|\\
&= |G(L)| + |1 - G(M - 1)| + \sum_{n = L + 1}^{M - 1}|G(n) - G(n - 1)|.
\end{align*}
Note that if $M - L = 1$, then the last sum is empty and should be interpreted as
equal to 0. Now, since $g(x)$ is increasing in $x\in [L,M-1]$,  
$\cot(\pi g(x))$ is decreasing in $x$ over the same interval. 
Hence, by \eqref{G_diff},
\[
|G(n) - G(n - 1)| = \frac{\cot(\pi g(n - 1)) - \cot(\pi g(n))}{2},
\]
for $L+1\le n \le M-1$. Therefore,
\begin{align}
\left|\sum_{n = L}^M e(f(n))\right| &=\sum_{n = L + 1}^{M - 1}\frac{1}{2}\left(\cot(\pi g(n - 1)) - \cot(\pi g(n))\right) + \left|\frac{1}{2} + \frac{i}{2}\cot(\pi g(L))\right|\notag \\
&\qquad\qquad + \left|\frac{1}{2} - \frac{i}{2}\cot(\pi g(M - 1))\right| \label{cglem_step1}\\
&= \frac{1}{2}\cot(\pi g(L)) - \frac{1}{2}\cot(\pi g(M - 1)) +
    \frac{1}{2}\frac{1}{\sin(\pi g(L))} + \frac{1}{2}\frac{1}{\sin(\pi g(M - 1))} \label{cglem_step2}\\
&= \frac{1}{2}\cot\left(\frac{\pi}{2}g(L)\right) + \frac{1}{2}\tan\left(\frac{\pi}{2}g(M - 1)\right) \label{cglem_step3}\\
&\le \frac{1}{2}\cot\left(\frac{\pi}{2U}\right) + \frac{1}{2}\tan\left(\frac{\pi}{2}\left(1 - \frac{1}{V}\right)\right) \label{cglem_step4}\\
&\le \frac{U}{\pi} + \frac{V}{\pi},\notag
\end{align}
as required. 
Going from \eqref{cglem_step2} to \eqref{cglem_step3}, 
we combined the various terms using the formulas  
$(1+\cos(\pi x))/\sin(\pi x)=\cot(\pi x/2)$
and $(1-\cos(\pi x))/\sin(\pi x)=\tan(\pi x/2)$, valid for
$0<x<1$. In \eqref{cglem_step4} we used the
monotonicity of $\cot(\pi x)$ and $\tan(\pi x)$ over $0<x<1/2$. In the last
line we used the inequalities 
$\cot(\pi x) < 1/(\pi x)$ 
and $\tan(\pi x/2) < 2/(\pi (1-x))$, valid for $0<x<1$.
Lastly, we note that passing from \eqref{cglem_step1} to
\eqref{cglem_step2} presents no difficulty if $M - L = 1$ since
$\cot(\pi g(L)) - \cot(\pi g(M - 1)) = 0$ in this case. 

\section{Proof of Lemma \ref{imp_imp_second_deriv_test}}\label{second_deriv_proof}
We divide the summation interval 
into about $2k$ suitable subintervals, chosen so that we may 
apply the generalized Cheng--Graham Lemma~\ref{cglem}
on about half of the subintervals and the trivial bound on the remaining
half. The special form of $g$ allows a sharper bound on $k$, 
so that in the definition of $\mu$ in Lemma \ref{imp_imp_second_deriv_test} 
the contribution of $\lambda$ can be reduced to $\lambda^{2/3}$.
Moreover, we adjust the definition of 
the boundary subintervals (determined by $C_0$ and $C_k$ below) 
to further reduce the number of sub-intervals in the sum. 
Recall that $r$ is a positive integer, $\lambda=(r+1)^3/r^3$, and 
$K \ge K_0 > 1$ where $K$ is an integer. 
We will use the following elementary inequality.

\begin{equation}\label{third_upper_bound}
\frac{rK}{rK + K - 1} = \frac{r}{r + (1 - 1/K)} \le \left(1 - K^{-1}\right)^{-1}\frac{r}{r + 1} \le \frac{1}{(1 - K_0^{-1})\lambda^{1/3}}.
\end{equation}
Let us recall that $W= \pi (r+1)^3K^3/t$
and the phase function $g(x)=f(x+m)-f(x)$ where $f(x)=\frac{t}{2\pi}\log(rK+x)$. 
We compute, for $m< L \le K$,  

\begin{align*}
    g'(L - 1 - m) - g'(0) &\le g'(K-1-m)-g'(0)\\ 
&= \frac{t}{2\pi r^3K^3}\cdot m\cdot \left(\frac{rK(K - 1 - m)}{rK + K - 1 - m}\right)\left(\frac{rK}{rK + m}\right)\left(1 + \frac{rK}{rK + K - 1}\right)\\
    &\le \frac{\lambda}{2W}\cdot m\cdot K\lambda^{-1/3} \cdot 1 \cdot \left(1 +
    \frac{1}{(1 - K_0^{-1})\lambda^{1/3}}\right),
\end{align*}
where in the first line we used that $g'$ is monotonically increasing, 
and in the last line we used the inequality \eqref{third_upper_bound},
as well as the observation
\[
\frac{rK(K - 1 - m)}{rK + K - 1 - m} \le \frac{rK\cdot K}{rK + K}=K
\lambda^{-1/3}.
\]
This observation follows since for any real number $\alpha$ the expression $\alpha
y/(\alpha + y)$
 is positive for positive $\alpha$ and $y$, as we have with $\alpha=rK$ and 
 $y\in [K-1-m,K]$, and is
increasing in $y$ away from the possible discontinuity at $y=-\alpha$.
Therefore, by definition of $\mu$, we obtain
\begin{equation}\label{gbound}
g'(L - 1 - m) - g'(0)\le  \frac{mK}{W}\mu.
\end{equation}
We next define
\[C_0 := \lfloor g'(0)\rfloor,\qquad C_k := \lfloor g'(L - 1 - m)\rfloor,\]
and let $\{g'(0)\} = \epsilon_1$ and $\{g'(L - 1 - m)\} = \epsilon_2$, where
$\{x\}$ denotes the fractional part of $x$. Let $\Delta$ be a number such that
$0<\Delta<1/2$, to be chosen later, and let
\begin{align*}
C_j &:= C_{j - 1} + 1,&& 1\le j\le k - 1,\\
x_j &:= \max\{(g')^{-1}(C_j - \Delta), 0\},&& 1\le j \le k,\\
y_j &:= \min\{(g')^{-1}(C_j + \Delta), L - 1 - m\},&& 0\le j \le k.
\end{align*}
So that
\begin{equation}\label{k_bound}
k = C_k - C_0 = g'(L - 1 - m) - \epsilon_2 - (g'(0) - \epsilon_1 ) \le
\frac{mK}{W}\mu + \epsilon_1 - \epsilon_2,
\end{equation}
where we used \eqref{gbound} in the last inequality.
Furthermore, since both $g'$ and $(g')^{-1}$ are increasing, we have, for $1 \le j \le k$,
\begin{align*}
y_j - x_j &= \min\{(g')^{-1}(C_j + \Delta), L - 1 - m\} - \max\{(g')^{-1}(C_j - \Delta), 0\}\\
&= (g')^{-1}(\min\{C_j + \Delta, g'(L - 1 - m)\}) - (g')^{-1}(\max\{C_j - \Delta, g'(0)\})\\
&\le 2\Delta|((g')^{-1})'(\nu_j)|,
\end{align*}
for some $\nu_j$ satisfying
\[
\max\{C_j - \Delta, g'(0)\} \le \nu_j \le \min\{C_j + \Delta, g'(L - 1 - m)\}.
\]
However this implies that $\xi_j := (g')^{-1}(\nu_j) \in [0, L - 1 - m]$ and hence 
\begin{equation}\label{interval_length_bound}
y_j - x_j \le 2\Delta|((g')^{-1})'(\nu_j)| = \frac{2\Delta}{|g''(\xi_j)|}\le \frac{2W\Delta}{m}.
\end{equation}
Next, we observe that $0\le x_1<y_1<x_2<y_2<\cdots<x_k<y_k\le L-m-1$, so the
$x_j$ and $y_j$ interlace. We treat the intervals $[x_1,y_1),
[x_2,y_2),\ldots,[x_k,y_k)$ using the trivial bound, and the intervals
$[y_1,x_2),[y_2,x_3),\ldots,[y_{k-1},x_k)$ using the Kusmin--Landau lemma.
As for the boundary intervals $[0,x_1)$ and $[y_k, L-m-1]$, they will require 
more careful analysis and a separate treatment.

By the triangle inequality, we have for any real numbers 
$a$ and $b$ such that $a\le b$, 
\[
\left|\sum_{a \le n \le b}e(g(n))\right| \le b - a + 1.
\]
Hence, applying \eqref{interval_length_bound},
\begin{align}
\left|\sum_{x_j\le n < y_j}e(g(n))\right| \le y_j - x_j + 1 \le \frac{2W\Delta}{m} + 1.\label{trivial_bound}
\end{align}
As for the complementary intervals $[y_j,x_{j+1})$, 
we have, by construction, $\|g'(x)\| \ge \Delta$ for all $x \in [y_j, x_{j + 1})$. 
So by Lemma \ref{cglem}, for $j = 1, \ldots k - 1$,
\begin{align}\label{kusmin-landaubound}
\left|\sum_{y_j \le n < x_{j + 1}}e(g(n))\right| \le \frac{2}{\pi \Delta}. 
\end{align}
It remains to consider the boundary intervals, starting with $[0,x_1)$. Let
\[
S_0 := \left|\sum_{0 \le n < x_1}e(g(n))\right|.
\]
We consider the following three cases. 
\subsection*{Case 1: $\epsilon_1 \in (1 - \Delta, 1]$}
Since $\epsilon_1 > 1 - \Delta$ and $(g')^{-1}$ is increasing,  
\[
(g')^{-1}(C_1 - \Delta) < (g')^{-1}(\lfloor g'(0)\rfloor + \epsilon_1) = (g')^{-1}(g'(0)) = 0,
\] 
hence $x_1 = 0$ and thus $S_0 = 0$. 
\subsection*{Case 2: $\epsilon_1 \in [0, \Delta)$}
Since $\epsilon_1 < \Delta$, 
\[
y_0 = (g')^{-1}(\lfloor g'(0)\rfloor + \Delta) > (g')^{-1}(g'(0)) = 0.
\]
Therefore, $\|g'(n)\| \ge \Delta$ for $n\in [y_0, x_1)$, so by Lemma \ref{cglem}
and on using the trivial bound to estimate the subsum over $[0,y_0)$ we obtain
\[
S_0 \le \left|\sum_{0 \le n < y_0}e(g(n))\right| + \left|\sum_{y_0 \le n < x_1}e(g(n))\right| \le \frac{W(\Delta - \epsilon_1)}{m} + 1 + \frac{2}{\pi \Delta}.
\]
Here, we also used the analog of \eqref{interval_length_bound} to bound the
length of the first subsum.
\subsection*{Case 3: $\epsilon_1 \in [\Delta, 1 - \Delta]$}
In this case,
$y_0 \le 0 \le x_1$ and 
$\ell+\epsilon_1 \le g'(n) \le \ell+ 1 - \Delta$ for all $n\in [0,
x_1)$, where $\ell=C_0$. So by Lemma \ref{cglem},
\[ 
S_0 \le \frac{1}{\pi}\left(\frac{1}{\epsilon_1} + \frac{1}{\Delta}\right). 
\]

Combining the three cases, we conclude that
\[
S_0 \le \frac{1}{\pi\Delta} + h(\epsilon_1),
\]
where 
\begin{equation}\label{h_def}
h(\epsilon) := \begin{cases}
    \displaystyle
\frac{W(\Delta - \epsilon)}{m} + 1 + \frac{1}{\pi\Delta} &\text{if } \epsilon \in [0, \Delta)\\
    \displaystyle
\frac{1}{\pi\epsilon} &\text{if }\epsilon \in [\Delta, 1 - \Delta]\\
    \displaystyle
-\frac{1}{\pi\Delta} &\text{if } \epsilon \in (1 - \Delta, 1]
\end{cases}
\end{equation}
Similarly, the sum corresponding to the other boundary interval $[y_k,L-1-m]$
satisfies 
\begin{equation}\label{Sk_bound}
S_{k} := \left|\sum_{y_k \le n \le L - 1 - m}e(g(n))\right| \le \frac{1}{\pi\Delta} + h(1 - \epsilon_2).
\end{equation}
The only difference in the treatment of $S_k$ is that if $\epsilon_2\in[\Delta,
1 - \Delta]$, then we use a slightly 
modified version of Lemma \ref{cglem} that holds for
sums over the closed interval $[a, b]$ instead of $[a, b)$. This is easy to
produce
since the bound from that lemma is independent of the length of summation.

Together with \eqref{trivial_bound} and \eqref{kusmin-landaubound}, 
and using \eqref{k_bound} to bound the sums over $k$, we obtain
\begin{align*}
\left|\sum_{0 \le n \le L - 1 - m}e(g(n))\right| &\le S_0 + \sum_{j = 1}^k\left|\sum_{x_j \le n < y_j}e(g(n))\right| + \sum_{j = 1}^{k - 1}\left|\sum_{y_j \le n < x_{j + 1}}e(g(n))\right| + S_{k}\\
&\le \frac{2kW\Delta}{m} + k +\frac{2(k - 1)}{\pi\Delta} + \frac{1}{\pi\Delta} + h(\epsilon_1) + \frac{1}{\pi\Delta} + h(1 - \epsilon_2)\\
&\le 2\left(\frac{mK\mu}{W} + \epsilon_1 - \epsilon_2\right)\left(\frac{1}{\pi\Delta} + \frac{W\Delta}{m} + \frac{1}{2}\right) + h(\epsilon_1) + h(1 - \epsilon_2).
\end{align*}
To balance the first two terms in the second factor, we choose 
\begin{equation}\label{Delta_choice}
\Delta = \sqrt{\frac{m}{\pi W}}.
\end{equation}
Note that if this choice of $\Delta$ 
is $\ge 1/2$, then the bound we are
trying to prove in Lemma~\ref{imp_imp_second_deriv_test} follows anyway since
the contribution of the first term of the bound will already be 
$4 \mu K \Delta \ge 2\mu K \ge K$, which 
is no better than the trivial bound.
Therefore, we may assume that
our choice of $\Delta$ satisfies $\Delta <1/2$.
Also, with our choice of $\Delta$ we have 
\begin{equation}\label{Delta_formula}
    \frac{W\Delta}{m} = \frac{1}{\pi \Delta}.
\end{equation}
Put together, on defining 
\begin{equation}\label{Hdef}
H(\epsilon) := \left(\frac{4}{\pi\Delta} + 1\right)\epsilon + h(\epsilon),
\end{equation}
and using \eqref{Delta_formula} to simplify in the second line next, we arrive at
\begin{align}
\left|\sum_{0 \le n \le L - 1 - m}e(g(n))\right| &\le
2\left(\frac{mK\mu}{W} + \epsilon_1 - \epsilon_2\right)\left(\frac{2}{\pi\Delta} + \frac{1}{2}\right) + h(\epsilon_1) + h(1 - \epsilon_2)\notag\\
&= \frac{4K\mu}{\sqrt{\pi W}}m^{1/2} + \frac{\mu K}{W}m - \frac{4}{\pi\Delta} - 1  
    + H(\epsilon_1) + H(1 - \epsilon_2).\label{mainbound}
\end{align}

We will show that for $\epsilon \in [0,1]$,
\begin{equation}\label{Hbound}
H(\epsilon) \le \frac{4}{\pi\Delta} - \frac{2}{\pi} + \frac{3}{2}.
\end{equation}
Substituting this bound back into \eqref{mainbound}, the lemma follows. 
We prove the inequality \eqref{Hbound} by considering three cases. 

\subsection*{Case 1: $\epsilon \in [0, \Delta)$} Then 
recalling \eqref{h_def} and using \eqref{Delta_formula}, 
\[
H(\epsilon) = \left(\frac{4}{\pi\Delta} + 1\right)\epsilon + \frac{W(\Delta -
\epsilon)}{m} + 1 + \frac{1}{\pi\Delta} = \left(1 +
\frac{4}{\sqrt{\pi}}\sqrt{\frac{W}{m}} - \frac{W}{m}\right)\epsilon + 1 + \frac{2}{\pi\Delta}.
\]
Noting that $1 + \frac{4x}{\sqrt{\pi}} - x^2 \le 1 + \frac{4}{\pi}$ for all $x$, we
have, since $0 \le \epsilon<\Delta < 1/2$,
\begin{equation}\label{case1bound}
H(\epsilon) < \left(1+\frac{4}{\pi}\right)\Delta + 1 + \frac{2}{\pi\Delta} \le
    \frac{3}{2} + \frac{2}{\pi} 
    + \frac{2}{\pi\Delta} < \frac{4}{\pi\Delta} -
    \frac{2}{\pi} + \frac{3}{2}.
\end{equation}
So the claimed bound on $H(\epsilon)$ in \eqref{Hbound} follows in this case.
\subsection*{Case 2: $\epsilon \in [\Delta, 1 - \Delta]$} 
Then,
\[
H(\epsilon) = \left(\frac{4}{\pi\Delta} + 1\right)\epsilon + \frac{1}{\pi\epsilon}.
\]
We observe that $H$ is convex over $x>0$ since $H''(x) > 0$. So,
\begin{equation}
H(\epsilon) \le \max\{H(\Delta), H(1 - \Delta)\}.
\end{equation}
On the other hand, since $\Delta < 1/2$,
\begin{equation}\label{Hdeltabound}
H(\Delta) = \frac{4}{\pi} + \Delta + \frac{1}{\pi\Delta} < \frac{4}{\pi} +
    \frac{1}{2} + \left(\frac{4}{\pi\Delta} - \frac{6}{\pi}\right) <
    \frac{4}{\pi\Delta} - \frac{2}{\pi} + \frac{1}{2}.
\end{equation}
Also,
\begin{align*}
H(1 - \Delta) &= \frac{4(1 - \Delta)}{\pi\Delta} + 1 - \Delta + \frac{1}{\pi(1 - \Delta)}.
\end{align*}
Since $x + \frac{1}{\pi x}$ is convex for $x>0$ and $1 - \Delta \in (1/2,1)$, we have 
\[
    1 - \Delta + \frac{1}{\pi(1 - \Delta)} \le \max\left(\frac{1}{2} +
    \frac{2}{\pi}, 1 + \frac{1}{\pi}\right) = 1 + \frac{1}{\pi}.
\]
Hence, 
\begin{equation}\label{H1deltabound}
H(1 - \Delta) \le \frac{4}{\pi\Delta} - \frac{3}{\pi} + 1.
\end{equation}
Combining \eqref{Hdeltabound} and \eqref{H1deltabound}, and using that
$-2/\pi+1/2$ and $-3/\pi+1$ are both $< -2/\pi+3/2$, the claim follows in this
case as well.
\subsection*{Case 3: $\epsilon \in (1 - \Delta, 1]$}Then, using $\Delta < 1/2$,
\begin{equation}\label{case3bound}
H(\epsilon) = \left(\frac{4}{\pi\Delta} + 1\right)\epsilon - \frac{1}{\pi\Delta}
    < \frac{4}{\pi\Delta} +1- \frac{2}{\pi} < \frac{4}{\pi\Delta}
    -\frac{2}{\pi}+\frac{3}{2},
\end{equation}
as claimed.

\section{Proof of Lemma \ref{imp_imp_third_deriv_test}}\label{third_deriv_proof}
In this lemma we propagate the improvements from Lemma \ref{imp_imp_second_deriv_test} through to the third derivative test. The approach remains the same as Hiary \cite{hiary_explicit_2016}. 
We first apply Weyl differencing from Lemma~\ref{weyldifflem} to the given
exponential, then we estimate
the differenced sums $s'_m(K)$. 
By Lemma \ref{imp_imp_second_deriv_test}, we have 
\begin{align*}
\left|s'_m(K)\right| &\le \frac{4K\mu}{\sqrt{\pi W}}m^{1/2} + \frac{\mu K}{W}m + 4\sqrt{\frac{W}{\pi}}m^{-1/2} - \frac{4}{\pi} + 2
\end{align*}
where $\mu$ is as defined in Lemma \ref{imp_imp_second_deriv_test}. We apply the inequalities
\[
\sum_{m = 1}^M\left(1 - \frac{m}{M}\right)\sqrt{m} \le \frac{4}{15}M^{3/2}, \qquad \sum_{m = 1}^M\left(1 - \frac{m}{M}\right)\frac{1}{\sqrt{m}} \le \frac{4}{3}\sqrt{M},
\]
\[\sum_{m = 1}^M\left(1 - \frac{m}{M}\right) \le \frac{1}{2}M, \qquad \sum_{m = 1}^M\left(1 - \frac{m}{M}\right)m \le \frac{1}{6}M^2,\]
which appear after \cite[Lemma 7]{cheng_explicit_2004}, which gives
\begin{align*}
\sum_{m = 1}^M\left(1 - \frac{m}{M}\right)|s'_m(K)| &\le \frac{16\mu K}{15\sqrt{\pi W}}M^{3/2} + \frac{\mu K}{6W}M^2 + \frac{16}{3}\sqrt{\frac{W}{\pi}}M^{1/2} + \frac{1}{2}\left(2 - \frac{4}{\pi}\right)M.
\end{align*}
Upon choosing $M = \lceil \eta W^{1/3}\rceil$, for some $\eta > 0$, we have $\eta W^{1/3} \le M \le \eta W^{1/3} + 1$ and
\begin{align*}
&\frac{2}{M}\sum_{m = 1}^M\left(1 - \frac{m}{M}\right)|s'_m(K)| \\
&\qquad\qquad \le \frac{32\mu K}{15\sqrt{\pi W}}M^{1/2} + \frac{\mu K}{3W}M + \frac{32}{3}\sqrt{\frac{W}{\pi}}M^{-1/2} + \left(2 - \frac{4}{\pi}\right)\\
&\qquad\qquad \le \frac{32\mu K}{15\sqrt{\pi W}}\sqrt{\eta W^{1/3} + 1} + \frac{\mu K}{3W}(\eta W^{1/3} + 1) + \frac{32}{3}\sqrt{\frac{W}{\pi}}\frac{1}{\sqrt{\eta W^{1/3}}} + \left(2 - \frac{4}{\pi}\right)\\
&\qquad\qquad = \frac{K}{W^{1/3}}\left[\frac{\eta\mu}{3W^{1/3}} + \frac{\mu}{3W^{2/3}} + \frac{32\mu}{15\sqrt{\pi}}\sqrt{\eta + W^{-1/3}}\right] \\
&\qquad\qquad\qquad\qquad + W^{1/3}\left[\frac{32}{3\sqrt{\pi \eta}} + \left(2 - \frac{4}{\pi}\right)W^{-1/3}\right]
\end{align*}
Substituting into Lemma \ref{weyldifflem}, we obtain the desired result.

\section{Concluding remarks}
It appears difficult to substantially improve the
constant $0.618$ in Theorem~\ref{main_theorem} 
without resorting to a large-scale numerical 
computation. Such a numerical computation would probably 
need to be extensive enough 
to allow us to, both, avoid using
the Riemann--Siegel-Lehman bound and increase the threshold 
value of $t_0$ where we start applying explicit van der Corput lemmas. 

We briefly describe two
theoretical approaches that may achieve modest improvements. 
Firstly, in Lemma
\ref{imp_imp_second_deriv_test} we used the inequality $\Delta < 1/2$, but this
can be replaced with a sharper inequality. This change will ultimately lower the constant term appearing in the main bound of Lemma \ref{imp_imp_second_deriv_test}. 

Secondly, our application of the third derivative test in
\eqref{final_piece_bound} may be inefficient on the last piece in the main sum
subdivision. This last piece contains $n_1 - RK+1$ terms, yet we always 
bound it as though it
contains $K$ terms. This could be wasteful if $n_1 - RK+1$ is much smaller than
$K$. More careful treatment of this boundary piece may produce savings for
certain ranges of $t$ in the ``intermediate" region, i.e. in the region $5.5\cdot 10^7 \le t < 10^{12}$.

Lastly, the context of this work highlights the 
sensitivity of explicit estimates in number theory to 
errors (even minor ones) that could compound quickly, and points to 
the need for a more automated approach in the future to verify 
explicit theoretical results.
Overall, the incorrect explicit Kusmin--Landau lemma 
affected all the published subconvex explicit estimates on zeta, as well as
other important explicit estimates in number theory. 
For example, the explicit $AB$ process derived in 
\cite[Lemma 1.2]{hiary_explicit_2016} was impacted, 
and in turn, as mentioned earlier, 
the constant $0.63$ in the
main theorem there was also impacted. 
Our work restores the $0.63$ constant and improves it. 
Therefore, we leave intact several works that have subsequently used 
the bound in \cite[Theorem 1.1]{hiary_explicit_2016}. 
\clearpage

\printbibliography
\end{document}